\documentclass{lms}
\usepackage[english]{babel}
\usepackage{amsmath}
\usepackage{amsfonts}
\usepackage{amssymb}
\usepackage{mathrsfs}
\usepackage{enumerate}
\usepackage[norelsize,linesnumbered,algoruled]{algorithm2e}
\usepackage{ifthen}
\newboolean{longversion}
\setboolean{longversion}{true}
\long\def\onlongversion#1{\ifthenelse{\boolean{longversion}}{#1}{\relax}}
\onlongversion{\usepackage{microtype}}
\newenvironment{myproof}[1][\proofname]{\proof[#1]}{\endproof}

\newtheorem{theorem}{Theorem}[section]

\newtheorem{lemma}[theorem]{Lemma}
\newtheorem{proposition}[theorem]{Proposition}

\newtheorem{remark}[theorem]{Remark}
\newtheorem{definition}[theorem]{Definition}

\makeatletter
\let\tilde@orig\tilde
\makeatother
\let\tilde\widetilde

\def\carac#1,#2{
\left[
\begin{smallmatrix}
#1 \\ #2
\end{smallmatrix}
\right]
}

\def\fonction#1#2#3#4#5{
  \begin{array}{ccccc}
    #1 & : & #2 & \to & #3 \\
       & & #4 & \mapsto & #5 \\
  \end{array}
}
\newcommand{\ideta}{I_{\Alg}}

\newcommand{\la}{\leftarrow}
\newcommand{\tilderho}{\tilde{\rho}}
\newcommand{\tildeK}{\tilde{K}}
\newcommand{\tildeA}{\tilde{A}}
\newcommand{\teta}{\tilde{\eta}}

\newcommand{\Gr}{\mathrm{Gr}}
\newcommand{\Alg}{\mathscr{K}}
\newcommand{\oGm}{\mathbb{G}_m}
\newcommand{\Sie}{\mathbb H} 
\newcommand{\Aff}{\mathbb A} 
\newcommand{\C}{\mathbb C} 
\newcommand{\N}{\mathbb N}
\newcommand{\Z}{\mathbb{Z}}

\newcommand{\R}{\mathbb R}

\newcommand{\Q}{\mathbb Q}
\newcommand{\proj}{\mathbb P}
\newcommand{\F}{\mathbb F}

\newcommand{\iso}{\stackrel{\sim}{\rightarrow}}
\newcommand{\xpol}{\mathscr{X}}
\newcommand{\pol}{\mathscr{L}}
\newcommand{\ppol}{\mathscr{L}_0}

\newcommand{\mpol}{\mathscr{M}}

\newcommand{\Sp}{\mathrm{Sp}}

\DeclareMathOperator{\Gal}{Gal}

\renewcommand*{\emptyset}{\varnothing}
\renewcommand*{\phi}{\varphi}
\renewcommand*{\epsilon}{\varepsilon}

\newcommand{\scalarmult}{\ensuremath{\mathrm{ScalarMult}}}

\newcommand{\normalize}{\ensuremath{\mathrm{Normalize}}}
\newcommand{\normaladd}{\ensuremath{\mathrm{NormalAdd}}}

\newcommand{\genericisogeny}{\ensuremath{\mathrm{GenericIsogeny}}}
\newcommand{\evaluate}{\ensuremath{\mathrm{Evaluate}}}

\renewcommand{\leq}{\ensuremath{\leqslant}}
\renewcommand{\geq}{\ensuremath{\geqslant}}

\DeclareMathOperator{\Id}{Id}

\DeclareMathOperator{\Spec}{Spec}
\DeclareMathOperator{\Var}{V}

\newcommand{\sO}{\tilde{O}}
\newcommand{\overk}{\ensuremath{\overline{k}}}

\newcommand{\tildeO}{\ensuremath{\tilde{0}_A}}
\newcommand{\tildeP}{\ensuremath{\tilde{P}}}

\newcommand{\tildeQ}{\ensuremath{\tilde{Q}}}
\newcommand{\tildePQ}{\ensuremath{\tilde{P+Q}}}

\newcommand{\overn}{\ensuremath{{n}}}

\newcommand{\overtwo}{\ensuremath{{2}}}

\newcommand{\Zstruct}[1]{\ensuremath{Z({#1})}}
\newcommand{\dZstruct}[1]{\hat{Z}(#1)}

\newcommand{\Zn}{\Zstruct{\overn}}

\newcommand{\dA}{\hat{A}}

\newcommand{\Ztwo}{\Zstruct{\overtwo}}

\newcommand{\dZtwo}{\dZstruct{\overtwo}}
\newcommand{\Ztwon}{\Zstruct{{2n}}}

\newcommand{\Kc}{\mathscr{B}}

\newcommand{\Rc}{\mathscr{R}}

\title[Computing separable isognies in quasi-optimal time]{Computing
  separable isogenies in quasi-optimal time}
\author{David Lubicz, Damien Robert}
\classno{14K02,14K25,11G10}

\begin{document}

\maketitle
\begin{abstract} 
  Let $A$ be an abelian variety of dimension $g$ together with a principal
  polarization $\phi: A \rightarrow \dA$ defined over a field $k$.
  Let $\ell$ be an odd integer prime to the characteristic of $k$
  and let $K$ be a subgroup of $A[\ell]$ which is maximal
  isotropic for the Riemann form associated to $\phi$. We suppose that
  $K$ is defined over $k$ and let $B=A/K$ be the quotient abelian
  variety together with a polarization compatible with $\phi$. Then $B$, as
  a polarized abelian variety, and the isogeny $f:A\rightarrow B$ are
  also defined over $k$. In this paper, we describe an algorithm that
  takes as input a theta null point of $A$ and a polynomial system
  defining $K$ and outputs a theta null point of $B$ as well as
  formulas for the isogeny $f$.
  We obtain a complexity of
  $\sO(\ell^{\frac{rg}{2}})$ operations in $k$ where $r=2$ (resp.
  $r=4$) if $\ell$ is a sum of two squares (resp. if $\ell$ is a sum of
  four squares)
  which constitutes an improvement over the algorithm described in
  \cite{robertcosset}. We note that the algorithm is quasi-optimal if
  $\ell$ is a sum of two squares
  since its complexity is quasi-linear in the degree of $f$.
\end{abstract}

  \section{Introduction}\label{sec:intro}

Let $k$ be a  field and let $A$ be a principally polarized
abelian variety of dimension $g$ defined over $k$.  Let $\ell$ be an
odd integer prime to the characteristic of $k$ and let $K$ be a
subgroup of exponent $\ell$ of $A$. Let $B=A/K$ be the quotient
abelian variety.  In this paper, we are interested in computing the
isogeny $f: A \rightarrow B=A/K$. Being able to compute isogenies
between abelian varieties has many applications in algebraic number
theory \cite{sutherland2009hilbert,kohel1996endomorphism,MorainIsogenies,bisson2009endomorphism,broker2009modularIsogenies,schoof1985elliptic,schoof1995counting,atkin1988number,elkies1992explicit}.

In order to have a concrete description of $A$, we consider a projective
embedding of $A$ provided by the global sections $H^0(A,\pol)$ of a
symmetric ample line bundle $\pol$. In the following, if $\xpol$ is an
ample line bundle on $A$, we denote by
$\phi_{\xpol}: A \rightarrow \hat{A}$ the polarization corresponding to
$\xpol$ and by $e_\xpol: A \times A \rightarrow \oGm$, the associated
Riemann form. We suppose that $\pol=\ppol^n$ with $\ppol$ a principal line
bundle of $A$ and $n \in \N$ that we call the level of $\pol$.  If $4 \mid
n$, we have a very convenient description of $A$ as the intersection of a
set of quadrics, given by the Riemann equations, in $\proj^{H^0(A,\pol)}$
the projective space over the $k$-vector space $H^0(A,\pol)$ of dimension
$n^g$. A choice of a basis of $H^0(A,\pol)$ and as a consequence of a choice
of an embedding of $A$ into the projective space $\proj^{n^g-1}$ is fixed
by a choice of a \emph{theta structure} for $(A,\pol)$ (see \cite[Def.
p. 297]{MR34:4269}).

We suppose that this embedding is defined over $k$ which implies that
$k$ contains the field of definition of $\phi_\pol$.  In order to
avoid field extensions and have a more compact representation of $A$,
we want to take $n$ as small as possible.  Most of the time, in
applications, $n=2$ or $n=4$. In the following, we assume that $2|n$
and that $\ell$ is prime to $n$.  

Now let $K$ be a subgroup of $A[\ell]$ maximal isotropic for the Riemann
form $e_{\pol^\ell}$, that is given as input by a set of homogeneous
equations in $\proj^{H^0(A,\pol)}$, and let $f: A \to B=A/K$ be the
corresponding isogeny. Then $\pol$ descends by $f$ to a line bundle $\mpol$ on
$B$ which is a $n$-power of a principal polarization $\mpol_0$. If $K$ is
$k$-rational, then both $B$ and the polarization $\mpol$ are also rational.
Our goal is to compute an embedding of $B$ in $\proj^{H^0(B,\mpol)}$ as
well as formulas for the isogeny $f:A\rightarrow B$.
We can prove
\begin{theorem}\label{main} Let $(A,\pol, \Theta_n)$ be a polarized
  abelian variety of dimension $g$ with a symmetric theta structure
  of level $n$ even. Let $\ell$ be an integer prime to $n$
  and assume that $\ell n$ is prime to the characteristic of $k$.
  Let $K$ be a subgroup of $A[\ell]$ maximal isotropic for
  $e_{\pol^\ell}$. Then one can compute the isogeny $A\rightarrow
  A/K$ in theta coordinates of level $n$ by using
  $\sO(\ell^{\frac{rg}{2}})$ operations in $k$ where $r=2$ (resp.
  $r=4$) if $\ell$ is a sum of two squares (resp. if $\ell$ is a sum of
  four squares).
\end{theorem}

This is the same result as \cite[Th. 1.1]{robertcosset} except that
they take for input the kernel $K$ given by generators of the group
$K(\overk)$, and the resulting theorem they have is that
the isogeny can be computed in $\sO(\ell^{\frac{rg}{2}})$ operations in $k'$
where $k'$ is the compositum of the field of definition of the geometric points of $K$.
When $k$ is a finite field, this yields a complexity
of $\sO(\ell^{\frac{rg^2}{2}})$ operations in $k$ for the algorithm of \cite{robertcosset}, complexity that can be much worse when $k$ is a number field.
In the case $r=2$, we remark that
the complexity of the algorithm presented in this paper is
quasi-linear in the degree of the kernel of the isogeny which is
quasi-optimal.

Our algorithm is very similar to the one of \cite[Th.
1.1]{robertcosset} which is based on one hand on the algorithm
described in \cite{MR2982438} to compute an isogeny $f:A \rightarrow
B$ between $A$ together with a line bundle of level $n$ and $B$ with a
line bundle of level $n \ell$, and on the other hand on the Koizumi
general addition formula \cite{MR0480543} from which it is deduced a
change of level formula (see \cite[Prop. 4.1]{robertcosset}). Our main
improvement consists in working with ``formal points'' rather than with
geometric points of the kernel $K$.
 
\onlongversion{%
\begin{remark}
  Let $\mpol$ be an ample line bundle on $B$ 
defining an embedding of $B$ into $\proj^{H^0(B,\mpol)}$. 
If we want to express the isogeny $f: A \to B$ in term of this embedding, 
it is natural to take $\mpol$ such that
$f^*(\mpol)$ is a power $\pol^m$ of $\pol$.
Indeed in that case $f$ comes
from a morphism of the projective spaces $\proj^{H^0(A,\pol^m)}
\rightarrow \proj^{H^0(B,\mpol)}$ which can be computed  without using
the equations defining the embedding of $A$ inside
$\proj^{H^0(A,\pol)}$. Descent theory tells us (see \cite[Prop.
2]{MR34:4269}) that, for $m \in \N^*$, there exists such a line bundle
$\mpol$ on $B$ if and only if $K$ is a
subgroup of $\ker \phi_{\pol^m}$ and is isotropic for
$e_{\pol^m}$.
As by \cite[Prop.  4]{MR34:4269}, $\ker \phi_{\pol^m} = \{ x \in
A(\overk) | m.x \in \ker \phi_\pol \}$, we have that $K$ is a
subgroup of $\ker \phi_{\pol^m}$ if and only if $\ell|m$. For
efficiency reasons, it is better to take $m=\ell$.
It is also more convenient to work with power of principal polarisations.
When $\ell$ is prime to $n$, by \cite[Prop. 2]{MR34:4269}, this is the case for $\mpol$ if and only if
$K$ is maximal isotropic in the $\ell$-torsion. 
This discussion should motivates the hypothesis made in
Theorem~\ref{main}.
  \label{rem:intro}
\end{remark}}
  
The paper relies heavily on the theory of theta functions which
provides a convenient framework to represent and manipulate global
sections of ample line bundles of abelian varieties.  In order to
avoid technical details, we have chosen to present our results
using the classical theory of theta functions.  For this, we assume
that $k$ is a number field and we suppose given a fixed embedding of
$k$ into $\C$.  Nonetheless, it should be understood that, by using
Mumford's theory of algebraic theta functions, all our algorithms
apply to the case of abelian varieties defined over any  field of
characteristic not equal to $2$. The notations used have been chosen
to make the translation into Mumford's formalism straightforward.

Our paper is organized as follows: in Section \ref{sec:basicfacts}, we
gather some basic definitions about theta functions. In Section
\ref{sec:recall}, we recall the principle of the algorithm of
\cite{robertcosset}. Then, in Section \ref{sec:gener}, we explain how to compute
with formal points in $K$. The Section \ref{sec:improve} is devoted to
the proof of the main results of this paper and in Section
\ref{sec:examples} we give some examples.

\section{Some notations and basic facts}\label{sec:basicfacts}
In this section, in order to
fix the notations, we recall some well known facts on analytic
theta functions (see
for instance \cite{MR85h:14026,MR2062673}). Let $\Sie_g$ be the $g$ dimensional Siegel upper-half
space which is the set of $g\times g$ symmetric matrices $\Omega$ with
coefficients in $\C$ whose
imaginary part is positive definite. For $\Omega \in \Sie_g$, we denote
by $\Lambda_\Omega= \Z^g+\Omega \Z^g$ the lattice of $\C^g$. If $A$
is a complex abelian variety of dimension $g$ with a
principal polarization then $A$ is analytically isomorphic to $\C^g /
\Lambda_\Omega$ for a certain $\Omega \in \Sie_g$.
For $a,b \in
\Q^g$, the theta function with rational characteristics $(a,b)$ is
the
analytic function on $\C^g \times \Sie_g$: \begin{equation}
  \theta \carac a,b (z,\Omega)= \sum_{\nu \in \Z^g} \exp\big[\pi i
  ^t(\nu+a)\Omega(\nu+a)+ 2\pi i ^t(\nu+a)(z+b)\big]. \end{equation} We say that a function $f$ on $\C^g$ is
$\Lambda_{\Omega}$-quasi-periodic of level $n \in \N$
if for all $z \in \C^g$ and $m \in \Z^g$, we have:$f(z+m)=f(z)$,
$f(z+\Omega m)=\exp(-\pi i n ^t{m}\Omega m - 2\pi i n {^t}{z} m) f(z)$
(where $^t v$ is the transpose of the vector $v$).
For any $n \in \N^*$, the set $H_{\Omega,n}$ of $\Lambda_\Omega$-quasi-periodic
functions of level $n$ is a $\C$-vector space of dimension $n^g$ a
basis of which can be given by the theta functions with
characteristics:
$(\theta \carac 0,{b/n} (z, n^{-1} \Omega))_{b\in [0, \ldots ,
n-1]^g}$. There is a well known correspondence (see \cite[Appendix
B]{MR2062673}) between
the vector space $H_{\Omega,n}$ and $H^0(A,\ppol^n)$ where $\ppol$
is the principal line bundle on $A$ canonically defined by a choice of
$\Omega$.

Once we have chosen a level $n \in \N^*$ and $\Omega \in \Sie_g$ such
that the abelian variety $A$ is analytically isomorphic to
$\C^g/\Lambda_\Omega$, for the rest of this paper, we adopt the
following conventions: we let $\Zn = (\Z / n \Z)^g$ and for a point
$z \in \C^g$ and $\nu \in \Zn$, we put $\theta^A_\nu(z) = \theta \carac
0,{\nu/n}(z,\Omega/n)$. If no confusion is possible, we will
write $\theta_\nu(z)$ in place of $\theta_\nu^A(z)$.

A theorem of Lefschetz tells that if $n \ge 3$, the functions in
$H_{\Omega,n}$ give a projective embedding of $A$:
\begin{equation}\label{eq:rhon}
\fonction{\rho_{n}}{\C^g/(\Z^g+\Omega\Z^g)}{\proj_\C^{\Zn}}%
{z}{\left(\theta_\nu(z)\right)_{\nu \in \Zn}}.
\end{equation}
For $\ell=2$, the functions in $H_{\Omega,2}$ do not give a projective
embedding of $A$. Actually, it is easy to check that for all $f \in
H_{\Omega,2}$, we have $f(-z)=f(z)$. 
Under some well known general conditions \cite[Cor. 4.5.2]{MR0480543}, the image of the embedding defined by $H_{\Omega,2}$ in
$\proj^{\Ztwo}$ is the Kummer variety associated to $A$, which is 
the quotient of $A$ by the automorphism $-1$.

It is natural to look for algebraic relations between theta
functions to obtain a description of the abelian variety as a closed subvariety
of a projective space. A lot of them are given by a result of Riemann (see
\cite[Th. 1]{RobLub}):
\begin{theorem}\label{th:thetaadd}
Let $i,j,k,l \in \Ztwon$. We suppose that $i+j$, $i+k$ and $i+l \in
\Zn$. Let $\dZtwo$ be the dual group of $\Ztwo$. For all $\chi \in
\dZtwo$ and $z_1, z_2 \in \C^g$ we have
\begin{equation*}\label{eq:thetaadd}
  \begin{split}
  \left(\sum_{\eta \in \Ztwo} \chi(\eta) \theta_{i+j+\eta}(z_1+z_2)
  \theta_{i-j+\eta}(z_1-z_2)\right) \left( \sum_{\eta \in \Ztwo}
  \chi(\eta)
  \theta_{k+l+\eta}(0) \theta_{k-l+\eta}(0) \right) \\ = \left( \sum_{\eta \in
  \Ztwo} \chi(\eta) \theta_{i+k+\eta}(z_1) \theta_{i-k+\eta}(z_1) \right) \left(
  \sum_{\eta \in \Ztwo} \chi(\eta) \theta_{j+l+\eta}(z_2) \theta_{j-l+\eta}(z_2)
  \right).
\end{split}
\end{equation*}
\end{theorem}

For $n \in \N^*$, we can associated to $\Omega \in \Sie_g$ its level
$n$ theta null point $(\theta_\nu(0))_{\nu \in \Zn}$.  By taking
$z_2=0$ in Theorem \ref{th:thetaadd}, we obtain a set of quadratic
equations which are parametrized by the theta null points of level
$n$.  A result of Mumford \cite[Th.2]{MR34:4269} tells us that if $4|n$ this system of
equations is complete in the sense that it gives the embedding of $A$
into $\proj^{\Zn}$ defined by $n$ and $\Omega$ following
\eqref{eq:rhon}.

In the context of Mumford's theory of algebraic theta function, the
data of $n$ and $\Omega$ which determines the theta null point is
replaced by a level $n$ \emph{theta structure} $\Theta_n$ \cite[Def.
p. 297]{MR34:4269}.  

Let $\kappa: \Aff^{\Zn} - \{ 0 \} \rightarrow \proj^{\Zn}$ be the
canonical projection.  We denote by $\tildeA$ the affine cone of $A$
that is the closed subvariety defined as the Zariski closure of
$\kappa^{-1}(A)$ in $\Aff^{\Zn}$.  We adopt the following convention:
if $P \in \proj^{\Zn}(\C)$, we will denote by $\tildeP$ an
\emph{affine lift} of $P$ that is an element of $\Aff^{\Zn}(\C)-\{0 \}$
such that $\kappa(\tildeP)=P$.  We denote by $\tildeP_\nu \in \C$
for $\nu \in \Zn$ the $\nu^{th}$ coordinate of the point $\tildeP$ and for
$\lambda \in \C$, we let $\lambda\star \tildeP \in \Aff^{n^g}(\C)$
be the point such that $(\lambda\star \tildeP)_\nu=\lambda \tildeP_\nu$.
In the same way, if $P \in A(\C)$, we denote by $z_P \in \C^g$ a point
such that $\rho_n(z_P)=P$. We remark that $z_P$ actually defines the
affine lift $(\theta_\nu(z_P))_{\nu \in \Zn} \in \tildeA(\C)$ of $P$
that we call a \emph{good lift} (note that such a good lift is not
unique since $z_P$ is defined by $P$ up to an element of
$\Lambda_\Omega$). We denote by $\tilderho_n : \C^g \rightarrow
\tildeA$, the map given by $z \mapsto 
(\theta_\nu(z))_{\nu \in \Zn}$.
In the following, we choose $\tildeO \in
\tildeA(k)$ an affine lift of $\kappa((\theta_\nu(0))_{\nu \in \Zn})$. 
  
As the canonical line bundle defined by $\Omega$ is not
  defined over $k$ (not even over an algebraic extension of $k$), we
  have, in general, $\tildeO= \lambda (\theta_\nu(0))_{\nu \in \Zn}$ for $\lambda
  \in \C$ not in $\overk$.  This subtlety does not change the
  projective embedding given by Riemann's equations, nor the
  computations presented in this paper for homogeneity reasons (see
  Proposition \ref{prop:hom} and \cite[Rem.  3]{RobLub}) so that we
  can safely suppose in the following that $\lambda =1$.  

Using Riemann equations, from the data of $\tildeO$,
$(\theta_\nu(z_1))_{\nu \in \Zn}$, $(\theta_\nu(z_2))_{\nu \in \Zn}$,
$(\theta_\nu(z_1-z_2))_{\nu \in \Zn} \in \tildeA(\C)$, one can recover
$(\theta_\nu(z_1+z_2))_{\nu \in \Zn}\in \tildeA(\C)$ provided that for sufficiently many
$\chi \in \dZtwo$, $k,l \in \Zn$, we have $\sum_{\eta \in \Ztwo}
\chi(\eta) \theta_{k+l+\eta}(0)\allowbreak \theta_{k-l+\eta}(0)\neq 0$. This is
always the case if the level $n$ is divisible by $4$, and if $n=2$, it is
true at the condition that the projective embedding given by level $2$
theta functions is projectively normal (see \cite[\S 4]{RobLub}). We will
always suppose that these conditions are fulfilled in the following.
The operation on affine points that we obtain is called a \emph{differential
addition} (see \cite{RobLub} for more details).
Chaining differential additions
in a classical Montgomery ladder \cite[Algorithm 9.5 p. 148]{MR2162716}
yields an algorithm that takes as inputs $\tildeQ=(\tildeQ_\nu)_{\nu \in
\Zn}$, $\widetilde{P+Q}=( (\widetilde{P+Q})_\nu)_{\nu \in \Zn}$,
$\tildeP=(\tildeP_\nu)_{\nu \in \Zn}$, $\tildeO=(\tilde{0}_\nu)_{\nu \in
\Zn}$ and an integer $\ell$ and outputs $\widetilde{Q+\ell P}$.
We write $\widetilde{Q+\ell P}=\scalarmult(\ell, \widetilde{P+Q}, \tildeP,
\tildeQ, \tildeO)$. 

If $4|n$, we can compute the "normal" addition law on the abelian
variety: actually, by computing sums of the form $\sum_{\eta \in
\Ztwo} \chi(\eta) \theta_i(z_1+z_2) \theta_{j_0}(z_1+z_2)$ with
Riemann relations for a fixed $j_0 \in \Zn$, from the knowledge of
$\tildeO$ and  the
projective points $(\theta_\nu(z_j))_{\nu\in \Zn}$ for $j=1,2$, we can
recover the projective point $(\theta_\nu(z_1+z_2))_{\nu \in \Zn}$.
We call $\normaladd$ this operation.

\section{Koizumi formula and isogeny computation}\label{sec:recall}

In this section, we briefly recall the principle of the isogeny
computation algorithm presented in \cite{robertcosset}. Let $(A,
\ppol)$ be a principally polarized abelian variety given by $\Omega
\in \Sie_g$ such that $A$ is analytically isomorphic to $\C^g /
\Lambda_\Omega$. Let $\ell$ be an odd integer number prime to $n$ and
let $K$ be a subgroup of $A[\ell]$ maximal isotropic for the Riemann
form $e_{\ppol}$. As $K$ is isotropic for $e_{\ppol}$, up to an action
of an element of the symplectic group $\Sp_{2g}(\Z)$ on $\Lambda_\Omega$, we
can always suppose that $K=\frac{1}{\ell} \Z^g/\Lambda_\Omega$ so that
our problem comes down to the computation of the isogeny:
$$\fonction{f}{A\simeq \C^g/\Lambda_\Omega}{B\simeq \C^g/\Lambda_{\ell
\Omega}}{z}{\ell z}$$

An important ingredient of the isogeny computation algorithm is the
following formula derived from the general Koizumi formula (see
\cite[Prop. 4.1]{robertcosset}):
\begin{proposition}\label{prop:koigen}
  Let $M$ be a matrix of rank $r$ with coefficients in $\Z$ such that $^t M M = \ell Id$. Let
  $X \in (\C^g)^r$ and $X =  Y M^{-1}\in (\C^g)^r$. Let $i \in
  (\Zn)^r$
  and $j = i M^{-1}$. Then we have
  \begin{equation}\label{eq:koizumi}
    \theta_{i_1}^B(Y_1) \ldots \theta_{i_r}^B(Y_r) = \sum_{\substack{t_1,
    \ldots, t_r \in \frac{1}{\ell} \Z^g/\Z^g \\ (t_1, \ldots, t_r)M =
(0, \ldots, 0)}} \theta_{j_1}^A(X_1 + t_1)
    \ldots \theta_{j_r}^A(X_r + t_r).
  \end{equation}

In particular the projective coordinates of the theta null point of $B$ are given by the equations
\begin{equation}\label{eq:thetanull1}
\theta_k^B(0)\ldots \theta_0^B(0)=\sum_{\substack{t_1,
    \ldots, t_r \in \frac{1}{\ell} \Z^g/\Z^g \\ (t_1, \ldots, t_r)M =
(0, \ldots, 0)}} \theta_{j_1}^A(t_1)
    \ldots \theta_{j_r}^A(t_r),
  \end{equation}
  where $j=(k,\ldots, 0)M^{-1} \in \Zn^r$.

  Likewise, if $P \in A(k)$ we can recover the projective coordinates of
  $f(P)$ via the equations
\begin{equation}\label{eq:thetanull2}
\theta_k^B(\ell z_P)\ldots \theta_0^B(0)=\sum_{\substack{t_1,
    \ldots, t_r \in \frac{1}{\ell} \Z^g/\Z^g \\ (t_1, \ldots, t_r)M =
(0, \ldots, 0)}} \theta_{j_1}^A(X_1+t_1)
    \ldots \theta_{j_r}^A(X_r +t_r).
  \end{equation}
where $z_P \in \C^g$ is such that $\rho_n(z_P)=P$,
$X=Y M^{-1}$ with $Y=(\ell z_P, 0,0,0)$ and $j=(k,0,\ldots, 0)M^{-1}$
for $k \in \Zn$. 
We remark that the $X_i$ are integral multiples $\alpha_i z_P$ of $z_P$,
where $(\alpha_1,\ldots,\alpha_r)=(1,\ldots,0) ^t M$ and $\alpha_i \in
\{1,\ldots, \ell-1\}$.
\end{proposition}

The algorithm takes as input the projective theta null
point of $A$, $0_A$ as well as a basis $(e_1, \ldots, e_g)$ of $K(\overk)$. We  can suppose that $(e_1, \ldots, e_g)$ is the
image by $\rho_n$ of the canonical basis of $\frac{1}{\ell} \Z^g/\Z^g$
if necessary by acting on $\Lambda_\Omega$ by an element of $\Sp_{2g}(\Z)$. 
In order to
compute the isogeny, we need to evaluate expressions of the form
of the right hand of (\ref{eq:koizumi}) which depends on the knowledge
of the good lifts of the form
$(\theta_{\nu}(X_i+t_i))_{\nu \in \Zn}$ where $X_i=\alpha_i z_P$ is an integral multiple of
$z_P$.
Moreover, we note that these good lifts have to be
coherent, meaning that they are all derived from the same $z_P$
such that $\rho_n(z_P)=P$ (we can not change $z_P$ from one term of
the sum to another).
Of course since everything is homogeneous, we only need to work up to a
common projective factor $\lambda$ and take coherent lifts over any affine
lift $\tildeP$ of $P$. This motivates the following definition:

\begin{definition}
  \label{def:compatible}
  We assume that we have fixed once and for all an affine lift
  $\tildeO$ of $0_A$.
  Suppose that we have fixed affine lifts $\tildeQ$ of every
  geometric point $Q \in K(\overk)$
  (we require that the lift of $0_A$ is $\tildeO$). 
  We say that these lifts are \emph{compatible} (with respect to $\tildeO$)
  if they differ from
  the good lifts $(\theta_{\nu}(z_Q))_{\nu \in \Zn}$ for $z_Q \in
  \frac{1}{\ell}\Z^g/\Z^g$ by the same projective factor $\lambda$
  (independently of $Q$).

  Let $P \in A(\overk)$ and 
  fix an affine lift $\tilde{P}$ above it.
  Suppose that together with the compatible lifts $\tilde{Q}$ from above
  we also have chosen affine lifts of every
  geometric point $\alpha P+Q$ where $Q \in K(\overk)$ and
  $\alpha \in \{0,\ldots,\ell-1\}$
  (we require that the lift of $P$ is $\tildeP$). 
  We say that these lifts are \emph{compatible with respect to
  $\tildeO$} if there exists
  $z_P \in \C^g$ with $\rho_n(z_P)=P$ such that they differ from the good lifts
  $(\theta_{\nu}(\alpha z_P+z_Q))_{\nu \in \Zn}$ for $z_Q \in 
  \frac{1}{\ell}\Z^g/\Z^g$ by the same projective factor $\lambda$.

We extend the definition of compatible lifts by saying that a set of points
$\{\tilde{Q}\}$ or $\{\tilde{P+Q}\}$ (where the points $Q$ are in $K(\overk)$
and $P$ is a fixed point in $A(\overk)$) are compatible (with $\tildeO$ or $\tildeP$) when they are part of a family of compatible points.
\end{definition}

We would like to have an algebraic way to identify compatible lifts.
\begin{proposition}\label{prop:hom}
  Let $Q \in K(\overk)$ and write $\ell=2 \ell'+1$.
  Let $\tilde{Q}$ be an affine lift of $Q$ compatible with $\tildeO$,
  then
  \begin{equation}
  \scalarmult((\ell'+1, \tildeQ, \tildeQ, \tildeO, \tildeO)=-
  \scalarmult(\ell',\tildeQ, \tildeQ, \tildeO, \tildeO).
    \label{eq:compat1}
  \end{equation}
  Reciprocally, if $\tilde{Q}$ satisfy equation~\eqref{eq:compat1}, we say
  that it is a potential compatible lift of $Q$ (with respect to $\tildeO$).
  
  Likewise, let $P \in A(\overk)$, $\tildeP$ be any affine lift of
  $P$, and let
  $\tilde{P+Q}$ be an affine lift of $P+Q$ compatible with $\tilde{P}$.
  Then
  \begin{equation}
    \scalarmult(\ell, \tilde{P+Q}, \tildeQ, \tildeP, \tildeO)=\tilde{P}.
    \label{eq:compat2}
  \end{equation}
  Reciprocally, if $\tilde{P+Q}$ satisfy equation~\eqref{eq:compat2}, we say
  that it is a potential compatible lift of $P+Q$ (with respect to $\tildeP$).
  \label{prop:pseudogood}
\end{proposition}

\begin{proof} We begin with the first claim.
  By \cite[Lem.  3.10]{MR2982438}, if we let $\chi(\alpha, \beta, m)
  =\scalarmult(m, \alpha \star \tildeQ,
  \alpha \star \tildeQ, \beta \star \tildeO, \beta \star \tildeO)$ for
  $\alpha, \beta
  \in \C$ and $m \in \N$, we have $\frac{\chi(\alpha,\beta,
  \ell'+1)}{\chi(\alpha,\beta,\ell')}=\frac{\alpha^{\ell}}{\beta^{\ell}}
  \frac{\chi(1,1,
  \ell'+1)}{\chi(1,1,\ell')}$. Thus, for reason of homogeneity, we can suppose that
  $\tilde{0}_A = (\theta_{\nu}(0))_{\nu \in \Zn}$. Let $z_Q \in
  \frac{1}{\ell}\Z^g$ be such that $\tilderho_n(z_Q)=\tildeQ$, we have $\ell'
  z_Q=-(\ell'+1) z_Q \mod \Z^g$. As 
$\scalarmult((\ell'+1, \tildeQ, \tildeQ, \tildeO,
  \tildeO)=(\theta_\nu((\ell'+1)z_Q))_{\nu \in \Zn})$ and
  $\scalarmult(\ell',\tildeQ, \tildeQ, \tildeO, \tildeO)=
  (\theta_\nu(\ell' z_Q))_{\nu \in \Zn})$, we obtain
\eqref{eq:compat1} because of the periodicity of $\theta_\nu$ with
  respect to $\Z^g$.
  
Still using \cite[Lem.  3.10]{MR2982438},
we have $\scalarmult(\ell, \alpha \star (\tilde{P+Q}), \beta \star
\tildeQ, \alpha \star \tildeP, \beta\star \tildeO)=\alpha \star
\scalarmult(\ell, \tilde{P+Q}, \tildeQ, \tildeP, \tildeO)$ for
$\alpha,\beta \in \C$. Thus
we can assume that
$\tilde{0}_A = (\theta_{\nu}(0))_{\nu \in \Zn}$ and
$\tildeP=\tilderho_n(z_P)$ for $z_P \in \C^g$. It is then easy to check, using
the $\Lambda_\Omega$-quasi-periodicity property of $\theta_\nu$ for
$\nu
\in \Zn$, that $\theta_\nu(z_P+\ell z_Q)=\theta_{\nu}(z_P)$.  \end{proof}

If $Q \in K(\overk)$, let $\tildeQ$ be any lift, and let $\lambda \in
\C$ be such that $\lambda \star \tilde Q$ is a compatible lift
$\tilderho_n(z_Q)$ for $z_Q \in \C^g$. We remark that because of
the symmetry relations \cite[Prop. 3.14]{MumfordTata1}, for all $\nu
\in \Zn$, we have $\theta_\nu((\ell'+1)z_Q)=\theta_{-\nu}((\ell'
z_Q))$.  Thus, if $\delta = (\tilde{\ell' Q})_\nu/(\tilde{(\ell'+1)
Q})_{-\nu}$ for any $\nu \in \Zn$, by \cite[Rem.  3]{RobLub}, we have
$\lambda^\ell=\delta$. Thus we can obtain $\lambda$ up to a $\ell^{th}$-root
of unity. In the same way, let $\tildePQ$ be any affine lift of $P+Q$
and $\mu \in \C$ be such that $\mu \star (\tildePQ)$ is a compatible
lift. Then from equation~\eqref{eq:compat2} and \cite[Lem.
3.10]{MR2982438}, if we set $\tilde{P + \ell Q}=\scalarmult((\ell,
\tilde{P+Q}, \tildeQ, \tildeP, \tildeO)$, we obtain a relation of the
form $\mu^\ell \lambda^{\ell(\ell-1)}=\beta$ where $\beta =
\frac{\tildeP_\nu}{(\tilde{P+\ell Q})_\nu}$, $\nu \in \Zn$. As we know
$\lambda^\ell$ from above, we can recover $\mu^\ell$.

We can summarize \cite[Prop.
18]{MR2824556} and the Section 3 of \cite{MR2982438} by the theorem:

  \begin{theorem}\label{th:mainold}
  Let $(e_1, \ldots, e_g)$ be a basis of a maximal isotropic subgroup
  $K$ of
  $A[\ell]$. Assume that we have chosen potential compatible lifts
  $\tilde{e_i}$, $\tilde{e_i+e_j}$ with respect to $\tildeO$. 
  Then:
  \begin{itemize}
  \item we can use the $\scalarmult$ algorithm to compute potential
    compatible
    lifts $\tilde{Q}$ for every point of $K(\overk)$ from the data of
    $\tilde{e_i}, \tilde{e_i+e_j}$;
  \item  up to an action of $\Sp_{2g}(\Z)$ on $\Lambda_\Omega$ which leaves $\Z^g \subset
    \C^g$ invariant (and also $(\theta_\nu(0))_{\nu \in \Zn}$),
    these
    lifts $\tildeQ$ are compatible with $\tildeO$.
  \item if $\tildeP$ is an affine lift of a point $P\in A(\overk)$ and
  we are given potential compatible lifts $\tilde{P+e_i}$ with respect to
  $\tilde{P}$, then they are actually compatible with $\tilde{P}$ and
  we can use the $\scalarmult$ algorithm to obtain compatible lifts (with
  respect to $\tildeP$) of all points of the form $\alpha P+Q$
  for $\alpha\in \{0,\ldots,\ell-1\}$.
  \end{itemize}
 \end{theorem}
\begin{myproof}[sketch]
  We prove the first two claims.
  Let $\lambda_i, \lambda_{ij} \in \C$ be such that 
 $\lambda_i \star \tilde{e_i}$, $\lambda_{ij} \star (\tilde{e_i
+e_j})$ are compatible lifts. 
We know by the discussion following Proposition~\ref{prop:pseudogood} that
the $\lambda_i, \lambda_{ij}$ are $\ell^{th}$-root of unity.
Using the transformation formula for theta functions \cite[p. 189]{MumfordTata1}, we can find a
symplectic matrix $M$ in $\Sp_{2g}(\Z)$ that leaves $K$ (globally) 
invariant and acts by $\star$ on the compatible lifts
of $e_i$ and $e_i+e_j$ exactly by $\lambda_i^{-1}$ and
$\lambda_{ij}^{-1}$. Since $\ell$ is prime to $2n$ we can even ask for $M$
to be congruent to the identity modulo $2n$ 
so that it leaves the theta null point
$\tildeO$ invariant \cite{MR48:3972}.
By definition of differential addition and Riemann equations (see Theorem
\ref{th:thetaadd}), starting from compatible
points $\tilde{e_i}$ and $\tilde{e_i+e_j}$, we can generate compatible
lifts of all geometric points in the kernel with the $\scalarmult$
algorithm.

As for the third claim, by homogeneity, we can assume that
$\tilde{P}$ is a good lift coming from a point $z_P \in \C^g$.  Let
$\mu_j\in \C$ be such that $\mu_j \star (\tilde{P+e_j})$ is a
compatible lift with respect to $\tilde{P}$. Then,
Proposition~\ref{prop:pseudogood} shows that $\mu_j$ are
$\ell^{th}$-root of unity. For $j=1, \ldots, g$, let $\epsilon_j \in
\frac{1}{\ell}\Z^g/\Z^g$ be such that $\rho_n(\epsilon_j)=e_j$.  The
functional equation for theta functions (\cite[p.  123]{MumfordTata1})
gives that for $a,b \in \Z^g$ and $j=1,\ldots, g$, we have
$\theta_\nu(z_P+\epsilon_j+\Omega a+b)=e^{- \pi i (n^2 {^t} a \Omega
a+2 n ^t a (z_P+\epsilon_j))} \theta_\nu(z_P)$.  So up to the common
constant $e^{-\pi i (n^2 {^t} a \Omega a+2n ^ta z_P)}$, if necessary
by changing $z_P$ to $z_P+\Omega a$ for a well chosen $a \in \Z^g$, we
can suppose that the $\tilde{P+e_i}$ are compatible lifts with respect
to $\tildeP$.  By using differential additions one can then generate
any compatible lift of the form $\alpha z_P+z_Q$, $z_Q \in
\frac{1}{\ell}\Z^g/\Z^g$.  \end{myproof}

For $i=1, \ldots, g$, we choose any affine lifts $\tilde{e_i}$ of
$e_i$.  By using normal additions, we can compute  $e_i+e_j$ and chose
affine lifts $\tilde{e_i+e_j}$. Let $\lambda_i, \lambda_{ij} \in \C$
be such that $\lambda_i \star \tilde{e_i}$, $\lambda_{ij} \star
(\tilde{e_i +e_j})$ are potential compatible lifts. Note that, by the
discussion after Proposition~\ref{prop:pseudogood}, we know
$\lambda_{i}^\ell$ and $\lambda_{ij}^\ell$.  Then,
Theorem~\ref{th:mainold}
tells us that by choosing any root for the $\lambda_i$ and
$\lambda_{ij}$ we can generate compatible lifts and evaluate
equation~\eqref{eq:thetanull1} to get the theta null point
$(\theta_\nu^B(0))_{\nu \in \Zn}$ of $B=A/K$.  Actually, if we work
formally with the $\lambda_i$, $\lambda_{ij}$, then the right hand of
(\ref{eq:thetanull1}) is a rational function of $\lambda_i$,
$\lambda_{ij}$. But by \cite[Lem. 4.2]{robertcosset}, this rational
function actually lies in $\C(\lambda_i^\ell, \lambda_{ij}^\ell)$ so
that we can evaluate it directly.

In the same way, we can compute the image $f(P)$ of $P \in A(\overk)$ given
by its projective theta coordinates. Indeed, 
from the knowledge of $P$ and $e_i$, we can compute the projective
points $P+e_i$ with normal additions. We choose affine lifts
$\tilde{P+e_i}$ of $P+e_i$,
and let $\mu_i \in \C$ be such that
$\mu_i \star (\tilde{P+e_i})$ are potential compatible lifts.
By Proposition~\ref{prop:pseudogood} we
know the value of $\mu_i^{\ell}$ and Theorem~\ref{th:mainold} tells us
that by
choosing any set of roots we get compatible lifts 
for the points appearing in the right hand of \eqref{eq:thetanull2}.
We can as such find the projective coordinates of $f(P)$.
Actually, if we work formally with the $\mu_i, \lambda_{ij}, \lambda_i$ we can evaluate the right hand of \eqref{eq:thetanull2}
as a rational function in
$\C(\mu_i,\lambda_i,\lambda_{ij})$. 
But by \cite[Lem. 4.4]{robertcosset} this rational function is an element of
$\C(\mu_i^\ell, \lambda_i^\ell, \lambda_{ij}^\ell)$ so that we can
evaluate it directly too.

\section{Equations for the kernel}\label{sec:genereqs}
We keep the notations of the preceding section.  Recall that $A$
together with a projective embedding inside $\proj^{\Zn}$ is given by
the data of its theta null point $\tildeO$.  As $K$ is defined over
$k$, we can suppose that $K$
is represented, as a $0$-dimensional subvariety of $\proj^{\Zn}$ by a
triangular system of homogeneous polynomial equations with
coefficients in $k$:
\begin{equation}\label{eq:system}
  \begin{aligned}
    Q_{i_1}(U_{i_0}, U_{i_1})&=&0\\
 Q_{i_2}(U_{i_0},U_{i_1},U_{i_2})&=&0\\
		      &\vdots&\\
  Q_{i_{n^g-1}}(U_{i_0}, U_{i_1}, \ldots, U_{i_{n^g-1}}) &=& 0,
  \end{aligned}
\end{equation}
for $i_j \in \Zn$. Indeed, from the knowledge of a set of generators
of a homogeneous ideal defining $K$ as a closed subvariety of
$\proj^{\Zn}$, such a triangular system may be obtained by computing
the reduced Groebner basis for the lexicographic order on the
variables $U_{i_0}, \ldots ,U_{i_{n^g-1}}$.

If necessary, by doing a linear change of variables, we can always
suppose that $i_0=0\in \Zn$ and $\Var(U_0)(\overk) \cap
K(\overk)=\emptyset$ (where $\Var(U_0)$ is the closed subvariety of
$\proj^{\Zn}$ defined as the zeros of $U_0$)
and we contend that this linear change of
variables is defined over a small extension of $k$. Indeed, we just
have to find a hyperplane of $\proj^{\Zn}$ passing through the origin
and avoiding all the points of $K$. We remark that the set of
hyperplanes passing through the origin and a point $x \in K(\overk)$
is represented by an hyperplane through the origin in the Grassmannian
$\Gr(n-1,\Zn)$. As a consequence, the set of hyperplanes passing
through the origin and a point of $K$ is a hypersurface $H$ of degree
bounded by $\ell^g=\#K(\overk)$ in $\Gr(n-1, \Zn)$.  If the field $k$
is infinite there exists a point of $\Gr(n-1, \Zn)(k)$ which is not in
$H(k)$ and we are done. If $k=\F_q$ is finite, by a result of Serre
\cite{MR1144337}, an upper bound for $\#H(k)$ is $\ell^g q^{n^g-2}+
\Pi_{n^g-3}$ where $\Pi_m = \frac{q^{m+1}-1}{q-1}$ is the cardinality
of $\proj^{m}(\F_q)$.  Thus, a random point in $\Gr(n-1,\Zn)(k)$ won't
be in $H(k)$ with high probability as soon as $q$ is sufficiently
large compared to $\ell^g$.  We deduce that we can find
(probabilistically) a hyperplane in $\proj^{\Zn}$ that avoids $K$ by
working over an extension of $k$ of degree $O(\ln(\ell^g))$.  Let $k'$
be such an extension, as the arithmetic in $k'$ has the same
complexity as the arithmetic in $k$ modulo a factor in
$O(\ln(\ell^{g}))$ which we have chosen to neglect in the statement of
Theorem \ref{main}, we can safely suppose in the following that
$k=k'$.

From our hypothesis, the variable $U_0$ plays the role of the
normalizing factor of a projective point. We consider the algebra
$$\Alg_0=k[U_i| i \in \Zn]/(Q_{i_1}(1,U_{i_1}), \ldots,
Q_{i_{n^g-1}}(1,U_{i_1}, \ldots, U_{i_{n^g-1}})).$$ We note that the
system (\ref{eq:system}) is generically such that
$\deg_{U_{i_1}}(Q_{i_1})=\ell^g$ and $\deg_{U_i} Q_i=1$ for $i \in \Zn
-\{0,i_1\}$. Actually, it can be seen exactly as before, by
considering the set of hyperplanes in the $\Gr(n-1, \Zn)$ that
intersect the vector defined by a pair of elements of $K(\overk)$,
that this property is always true if we do a linear change of
coordinates. This linear change of coordinates may involve an
extension of $k$ of degree $O(\ln(\ell^{g}))$ when $k$ is
finite with negligible consequences for the asymptotic complexity of
the arithmetic of the base field. From now on, we suppose that $\deg_{U_i} Q_i=1$ for $i\in
\Zn-\{0,i_1\}$ and we let $Q(U)=Q_{i_1}(1,U) \in k[U]$.  The algebra
$\Alg_0$ is then isomorphic to the algebra $\Alg=k[U]/(Q)$.

The transformation from equation~\eqref{eq:system} to a polynomial
system defined by one polynomial $Q$ will not be necessary for the
algorithms presented in the next section, but it helps for
computations in the algebra associated to $\Alg$.

\section{Computation with formal points}\label{sec:gener}

We keep the notations of the previous section, we thus have an
isomorphism $K \iso \Spec \Alg$ associated to the isomorphism $\Alg
\iso \Alg_0$ on coordinate functions.  We recall that a point $\eta
\in A(\Alg)$ is by definition a morphism $\eta: \Spec \Alg \rightarrow
A$. We call such a point a \emph{formal point}, by opposition to
the geometric points in $A(\overk)$.  
Since $\Alg$ is \'etale,
a formal point $\eta \in A(\Alg)$ is given by the
data of a
$\Gal(\overk/k)$-equivariant morphism (of sets) $K(\overk) \to
A(\overk)$.  We denote by $I_{\Alg} \in A(\Alg)$ the point coming from
the closed immersion $K \rightarrow A$ defined by \eqref{eq:system}
which can be seen as a point in $K(\Alg)$. 
For instance, if
  $K \setminus \{0_A\}$ is irreducible then $I_{\Alg}$ restricts to its generic point.
Since $A$ is an abelian variety, $A(\Alg)$ is an abelian group.  Let
$P \in K(\overk)$ be a geometric point $P: \Spec(\overk) \rightarrow
\Spec(\Alg)$. For a formal point $\eta \in A(\Alg)$, we denote by
$\eta(P) \in A(\overk)$ the geometric point $\eta \circ P:
\Spec(\overk) \rightarrow \Spec(A)$ obtained by "specialisation".
In the same way, if $x \in \Alg$ and $P \in K(\overk)$, we denote by
$x(P)$ the value of $x$ in $P$.

Let $Q=\prod_{i \in I} R_i$, for $R_i \in k[U]$, be a
decomposition of $Q$ in irreducible elements.  Via the projective embedding $A \to
\proj^{\Zn}$ a formal point $\eta \in A(\Alg)$ can be seen as a
projective point $\eta \in \proj^{\Zn}(\Alg)$.  Since $\Alg$ is 
an \'etale algebra
such a projective point is given by the data of a
$(\eta_\nu)_{\nu \in \Zn} \in \Alg^{\Zn}$ such that  for all $i \in I$
at least one of the $\eta_\nu$ is invertible modulo $R_i$; modulo the
action of invertible elements of $\Alg$ on $\Alg^{\Zn}$.  The
isomorphism $K \iso \Spec \Alg$ shows that the projective coordinates
of $\eta \in A(\Alg)$ are given by $(1, \eta_{i_1},
Q_{i_2}(\eta_{i_1}), \ldots, Q_{i_{n^g-1}}(\eta_{i_1}))$ where the
$Q_{i_j} \in k[U]$ are defined in Section~\ref{sec:genereqs}.  The formal point $I_{\Alg} \in A(\Alg)$ is such
that $(I_{\Alg})_{i_1}=U$. As for geometric points, we denote by
$\teta\in \tildeA(\Alg)$ an affine lift of $\eta \in A(\Alg)$ and let
$(\teta)_i \in \Alg$ for $i \in \Zn$ be its corresponding affine
coordinates.  A point $P \in A(k)$ corresponds to a formal
point $\eta_P \in A(\Alg)$ via the constant morphism $Q \in
\Alg(\overk) \mapsto P$. We will often make this identification in the
following. 

 The arithmetic of geometric points recalled in Section
 \ref{sec:basicfacts} translates {\it mutatis mutandis} into
 arithmetic with formal points.  For instance, from the knowledge of
 $\teta_1,\teta_2 \in \tildeA(\Alg)$, and $\tilde{\eta_1-\eta_2} \in
 \tildeA(\Alg)$, one can compute the differential addition
 $\tilde{\eta_1+\eta_2}\in \tildeA(\Alg)$. This can be proved exactly
 in the same way as with geometric points using Riemann's equations
 and we obtain the same formulas. The case $g=1$ is given for instance
 in \cite{RobLub}. We remark that in order to be able to compute the
 differential addition, we have to find an invertible coordinate
 $(\tilde{\eta_1-\eta_2})_{\nu_0} \in \Alg$ for $\nu_0 \in \Zn$.  But in
 fact, there is always at least one such coordinate modulo $R_i$ for
 every $i \in I$. So we can
 always compute the differential addition, provided that we work
 modulo $R_i$, which is the case if we know the factorisation of $Q$, the
 polynomial defining $\Alg$. Actually we don't even need to compute
 the factorisation before hand because when we try to invert a non
 zero element, if it is non invertible the Euclidean algorithm will give
 a factor of $Q$.  We remark that when the difference is a formal
 point with value in $K$, then, by the hypothesis made in
 Section~\ref{sec:genereqs} on the equations, the coordinate $i=0$ is
 always invertible which helps in the computations.  In the same way,
 if $4 \mid n$, one can compute normal addition of generic points.
 Addition and multiplication operations in $\Alg$ take $\sO(\ell^{g})$
 operations in $k$.  Computing the inverse of an element of $\Alg$ can
 be done in $\sO(\ell^{g})$ operations in $k$ via the extended
 Euclidean algorithm. Thus, differential and normal additions of elements
 of $\tildeA(\Alg)$ take $\sO(\ell^{g})$ operations in $k$.

For $P \in K(\overk)$, $\eta \in A(\Alg)$ and $\teta \in
\tildeA(\Alg)$ any affine lift of $\eta$, we denote by $\teta(P) \in
\tildeA(\overk)$ the point with affine coordinates $(\teta)_i(P)$.  As
two formal points $\teta, \teta' \in \tildeA(\Alg)$ are equal if and
only if for all $P \in K(\overk)$, $\teta(P)=\teta(P)$, this allows to
verify certain properties by specializing a formal point to geometric
points. For instance, the result of a chain of differential addition
that one can compute from a set $\teta_1, \ldots, \teta_m \in
\tildeA(\Alg)$ of affine lifts of $\eta_1, \ldots, \eta_m \in A(\Alg)$
does not depend on the order of the operations because, as proved in
\cite[Cor.  3.13]{MR2982438}, it is true for geometric points.
Also if $\eta \in K(\Alg)$ then $\ell \eta=0_A$.

Let $\eta=P+\eta_K \in A(\Alg)$ with $\eta_K \in K(\Alg)$ and let
$\teta \in \tildeA(\Alg)$ be an affine lift. We say that $\teta$ is a
compatible lift of $\eta$ (relative to $\tildeP$) if for all $Q \in
K(\overk)$, $\teta(Q) \in \tildeA(\overk)$ is a compatible lift of
$P+\eta_K(Q)$ relative to $\tildeP$.  It is easy to extend
Proposition~\ref{prop:pseudogood} to compute compatible lifts of
formal points.

\begin{proposition}
  \label{prop:pseudogoodformal}
  Let $\eta_K \in K(\Alg)$, let $\teta_K \in \tildeA(\Alg)$ be any affine lift.
Write $\ell=2\ell'+1$, and compute $\ell'
\teta_K=\scalarmult(\ell',\teta_K, \teta_K, \tildeO, \tildeO)$, $(\ell'+1)
\teta_K=\scalarmult(\ell'+1, \teta_K, \teta_K, \tildeO, \tildeO)$. 
Let $\lambda \in \Alg_{\C}$ (where $\Alg_{\C}=\Alg \otimes_k \C$) be
an invertible element,
then $\lambda \star \teta_K$ is a potential compatible lift
if and only if
\begin{equation}
 ((\ell'+1) \teta_K)_\nu \lambda^\ell - (\ell' \teta_K)_{-\nu}=0,
  \label{eq:equationscompat1}
\end{equation}
for $\nu \in \Zn$.

Likewise, let $\eta=\eta_K+P \in A(\Alg)$ where $\eta_K \in K(\Alg)$
and $P \in A(k)$.
Fix affine lifts $\tildeP \in
\tildeA(k)$ of $P$ and $\teta_K \in \tildeA(\Alg)$ of $\eta_K$ and denote by $\teta \in \tildeA(\Alg)$ 
a compatible lift of $\eta$.
Let $\teta^0=\scalarmult(\ell,
  \teta, \lambda \star \teta_K, \tildeP, \tildeO) \in
  \tildeA(\Alg(\lambda))$.
  Then modulo the equations from~\eqref{eq:equationscompat1}, $\teta^0$ is in $\tildeA(\Alg)$ and $\mu \in \Alg_{\C}$ is such that
$\mu \star \teta$ is a potential compatible lift (relative to $\tildeP$)
if and only if
\begin{equation}
  \mu^{\ell}\teta^{0}_{\nu} - \tildeP_\nu=0
  \label{eq:equationscompat2}
\end{equation}
for $\nu \in \Zn$.
\end{proposition}

\begin{proof}
By equation~\eqref{eq:compat1} and \cite[Rem.  3]{RobLub}, we have that
$\lambda \star \teta$ is a potential compatible lift if and only if
$\lambda^{(\ell'+1)^2}\star( (\ell'+1) \teta)= -\lambda^{{\ell'}^2}\star
(\ell' \teta)$. We
thus obtain that if $\lambda$ is an element of an \'etale extension of $\Alg$
which satisfy the equations~\eqref{eq:equationscompat1} then 
$\lambda\star \teta$ is a potential compatible lift.

By \cite[Rem.  3]{RobLub}, the coordinates of $\teta^0$ only have factors
of the form $c \lambda^{\ell(\ell-1)}$ where $c \in \Alg$. By looking at the
equations~\eqref{eq:equationscompat1}, it is thus clear that $\teta^0$ does not
depend on the choice of a compatible lift for $\eta_K$.
By equation~\eqref{eq:compat2} and \cite[Rem.  3]{RobLub} again, we have that
$\mu \star \teta$ is a potential compatible lift if and only if
$\mu^{\ell} \star \teta^0=\tildeP$. Thus,
if $\mu$ is a root of the polynomials from
equations~\eqref{eq:equationscompat2}
in an \'etale extension of $\Alg$
then $\mu \star \teta$ is a potential compatible lift.
\end{proof}

We denote by $\normalize(\teta,\tildeP)$ the algorithm which 
outputs
equations~\eqref{eq:equationscompat2} defining
$\mu$ such that $\mu \star \teta$ is a compatible lift (relative to $\tildeP$).
It is clear
from its description that  $\normalize$ applied to a
formal point takes $\sO(\ell^{g})$ operations in $k$.
In practice, given the hypothesis about the kernel made in
Section~\ref{sec:genereqs}, it suffices to use
equation~\eqref{eq:equationscompat1} with coordinate $\nu=0$ to
determine $\lambda^\ell$, and it suffices to use
equation~\eqref{eq:equationscompat2} with a coordinate $\nu$ such that
$P_\nu \ne 0$ to determine $\mu^{\ell}$. The corresponding algorithm
(with $\nu=0$) for \normalize{} is given in
Algorithm~\ref{algo:normalize}.

\begin{algorithm}
\SetKwInOut{Input}{input}\SetKwInOut{Output}{output}
\SetKwComment{Comment}{/* }{ */}
\Input{
  \begin{itemize}
    \item $\eta=\eta_K+P$ where $\eta_K \in K(\Alg)$ and $P \in A(\overk)$;
    \item $\tildeP$ an affine lift of $P$ and $\teta$ an affine lift of
      $\eta$.
  \end{itemize}
}
\Output{
  An equation $\mu^{\ell}=c$ for $c \in \Alg$
  so that $\mu \star \teta$ is a
  compatible lift with $\tildeP$ when $\mu$ satisfy this equation.
}
\BlankLine
$\mu_K^{[1]} \la \scalarmult(\ell'+1,\lambda \star \teta_{K},\lambda \star \teta_K,\tildeO,\tildeO)$ and
$\mu_K^{[0]} \la \scalarmult(\ell',\lambda \star \teta_{K},\lambda \star \teta_K,\tildeO,\tildeO)$ 
where $\ell=2 \ell'+1$ \;
$\mu^{[0]} \la \scalarmult(\ell,\mu \star \teta,\lambda \star \teta_K,\tildeP,\tildeO)$ \; 
\Return $\mu^\ell=\frac{\tildeP_0}{\mu_0^{[0]}}
\left(\frac{{\mu_{K,0}^{[0]}}}{{\mu_{K,0}^{[1]}}}\right)^{\ell-1}$ \;
\caption{Algorithm $\normalize$}
\label{algo:normalize}
\end{algorithm}

\begin{remark}
  Compared to Proposition~\ref{prop:pseudogoodformal},
  Theorem~\ref{th:mainold}
starts with a basis of potential compatible lifts $\tilde{e}_i$ and
derives compatible
lifts for the set of all geometric points of the kernel.
By contrast, if $\teta$ is a compatible lift of a formal point $\eta
\in A(\Alg)$ and $\tildeP$ an affine lift of $P=\eta(0)\in A(\overk)$, then by
definition all the $\teta(Q), Q\in K(\overk)$ are compatible with $\tildeP$, but they won't be in general globally compatible with each others.
  \label{rem:diffwithgeompoints}
\end{remark}

One can think of the formal point approach that we have presented as
doing formally computations with geometric points. \onlongversion{We end up this
section by explaining that it is actually possible to follow exactly
the same procedure as that of the algorithm of \cite{robertcosset} to
compute isogenies with formal points.  Let $k'$ be the compositum of
the fields of definition of elements of $K(\overk)$. The Galois group
$\Gal(k'/k)$ acts on $K(\overk)$. As the group law of $A$ is defined
over $k$, this action is linear and $\Gal(k'/k)$ acts on the elements
of $A(\Alg)$.  Suppose that $K \setminus \{0_A\}$ is irreducible and that $\Gal(k'/k)$
is cyclic generated by $g$. Let $\eta_1=\Id_{\Alg} \in A(\Alg)$ and
let $\eta_i=g^i \eta_1$ for $i=2, \ldots, g$.  Then, as by hypothesis
$g$ generates $\Gal(k'/k)$, $\eta_1, \ldots, \eta_g$ are linearly
independent formal points of $K$. One can compute $\eta_i +\eta_j$
with normal additions, then compute potential compatible lifts of all
the $\eta_i$ and $\eta_i+\eta_j$ and use chain of differential
additions to obtain potential compatible lifts of all the points of the
kernel. Thus, it is possible to evaluate the right hand of
(\ref{eq:thetanull1}) with formal points: the result of the
computations in $\Alg$ will actually lie in $k$. We remark that the
algorithm that we have just sketched is just a fancy way to compute
with the splitting field defined by $Q$. This naive approach does not
improve the complexity of the algorithm of \cite{robertcosset} even in
the favorable case that we have considered in the paragraph.}

\section{An algorithmic improvement}\label{sec:improve}

We come back to the context of Section \ref{sec:recall}, let $(A, \pol_0)$ be a
principally polarized abelian variety given by $\Omega \in \Sie_g$. We
have seen that the general isogeny computation problem boils down to
the case where $K=\frac{1}{\ell} \Z^g/\Lambda_\Omega$ and 
\begin{equation}\label{eq:iso}
\fonction{f}{A\simeq
\C^g/\Lambda_\Omega}{B\simeq \C^g/\Lambda_{\ell \Omega}}{z}{\ell z}
\end{equation}

We explain how to evaluate the
expression of Proposition \ref{prop:koigen} with formal points and
derive an efficient algorithm.
Actually, by looking at the right hand term of
equation~\eqref{eq:thetanull1},
we need to deal in a ``formal'' way with $r$-tuples of the form $X_i+t_i$.
This motivates the following definition.

\begin{definition}
  For $t$ a positive integer, 
let $\Kc=\Alg^{\otimes t}$ so that $K^t \iso \Spec \Kc$. We define a 
\emph{formal tuple} as a point $\eta \in A(\Kc)$.

For $i=1, \ldots, t$, let $\mu_i : K^t \rightarrow K$ be the $i^{th}$
projection. Note that $\mu_i$ induces the natural
injection of coordinate algebras given by $x \mapsto 1\otimes \ldots \otimes
x\otimes \ldots \otimes 1$ where $x$ is in the $i^{th}$ position.

For $i=1,\ldots, t$, we denote by $\ideta^{(i)}=\ideta \circ \mu_i \in
A(\Kc)$ so that $\sum_{i=1}^t \ideta^{(i)}$ (the sum is the group law
of $A(\Kc)$) is the point coming from the canonical morphism associated
to the addition: $A^t \to A$.
\label{def:linearformaltuple} \end{definition}

Since $\Kc$ is also an \'etale algebra,
everything said in Section~\ref{sec:gener} about the arithmetic of formal
points (differential additions, $\normalize$\ldots) also apply to formal tuples.
In particular $\eta \in A(\Kc)$ is represented by its coordinates
$\eta_\nu \in \Kc$ for $\nu \in \Zn$,
addition and multiplication operations in the algebra $\Kc$ 
take $\sO(\ell^{tg})$ operations in $k$ and computing the
inverse of an element of $\Kc$ can be done in $\sO(\ell^{tg})$ operations
in $k$ via a recursive use of the extended Euclidean algorithm.  

\onlongversion{%
\begin{remark}
A point $\eta \in A(\Kc)$ that is also a morphism of
algebraic group is necessarily of the form
$\eta^{(1)}\circ \mu_1+\dots+\eta^{(t)}\circ \mu_t$ for
a tuple $(\eta^{(1)},\ldots, \eta^{(t)})$ of
points in $A(\Alg)$. The formal point $\eta^{(i)}$ can be recovered from
the formal tuple $\eta$ via $\eta^{(i)}=\mu'_i \circ \eta$ where $\mu'_i$
correspond to the canonical inclusions of $K$ into $K^t$.
In practice we will be working with formal tuple coming from
linear combinations
$P+\sum \lambda_i I^{(i)}_{\Alg}$ where $P \in A(k)$.
\end{remark}}

Let $M$ be a $r\times r$ matrix with integer coefficients such that $^tM M = \ell Id$. 
Write $\ell=\ell_1 \ell_2$ where $\ell_1$ is the biggest square factor of
$\ell$. If $\ell =\ell_1$, we can take $M=(\sqrt{\ell_1})$ and fix
$r=1$.
If 
$\ell_2 \neq 1$ and all prime factors of $\ell_2$ are congruent to 
$1 \mod 4$,
then there exists $a,b \in \N^*$ such that
$\ell_2=a^2+b^2$ thus we can take $M=\sqrt{\ell_1} M_0$, with $M_0=\left( \begin{smallmatrix} a & b \\ -b & a
\end{smallmatrix}
\right)$ and fix $r=2$. Finally, if there is a prime factor of $\ell_2$ congruent to  $3 \mod 4$,
we can write
$\ell_2=a^2+b^2+c^2+d^2$ for $a,b,c,d \in \N$ such that $a^2+b^2
\not\equiv 0 \mod \ell_2$, take for $M_0$ the matrix of multiplication by
$a+ib+cj+dk$ in the quaternion algebra over $\R$, $M=\sqrt{\ell_1} M_0$ and fix $r=4$.
We consider the isogeny of algebraic groups $F:K^r\rightarrow K^r$ acting componentwise by
the matrix $M$.
Denote by $\ker F$ the kernel subvariety of $F$. 
In the case that $\ell =\ell_1$, $\ker F$ is isomorphic to $\sqrt{\ell_1}
K$, an isomorphism being given on geometric
points by the identity.  In
the case that
$\ell_2 \neq 1$ is a sum of two squares (resp. a sum of four squares),
$\ker F$ is isomorphic to $L=\sqrt{\ell_1} K$
(resp. $L=\sqrt{\ell_1} K^2$), an isomorphism $L \rightarrow \ker F$
being given on geometric points by $x \mapsto (x, \beta_0 x)$ with
$\beta_0=-b/a \mod \ell$ (resp. $(x_1, x_2) \mapsto (x_1, x_2, \alpha
x_1+\beta x_2, \gamma x_1+\delta x_2)$ with $\left( \begin{smallmatrix}
\alpha & \beta \\ \gamma & \delta \end{smallmatrix} \right) = 
\left( \begin{smallmatrix} a & b \\ -b & a \end{smallmatrix}
\right)^{-1}
\left( \begin{smallmatrix} c & d \\ d & -c \end{smallmatrix} \right)
=\frac{1}{a^2+b^2} \left( \begin{smallmatrix} ac-db & ad+bc
\\ ad+bc & bd-ac \end{smallmatrix}\right)$, note that 
$\left( \begin{smallmatrix} a & b \\ -b & a
\end{smallmatrix} \right)$ is invertible modulo $\ell$ since $a^2+b^2
\not\equiv
0 \mod \ell$).
With these notations, we have:

\begin{proposition}\label{prop:main1}
Let $P \in A(k)$ and
let $z_P \in \C^g$ be such that $\rho_n(z_P)=P$. 
Fix $k \in \Zn$ and let $j=(k,\ldots, 0)M^{-1} \in \Zn^r$.
Let $N$ be the
cardinal of the kernel of the group morphism 
$K^t(\overk) \rightarrow \sqrt{\ell_1} K^t(\overk), 
x \mapsto \sqrt{\ell_1} x$.
If $\ell_2 = 1$, set $t=1$ else set $t=r/2$ and let $\Kc=\Alg^{\otimes t}$.
Moreover :
\begin{itemize}
  \item If $\ell_2=1$, let $\eta^{(1)} = \sqrt{\ell_1} \ideta+\zeta P$ for $\zeta \in \Z$. Let 
$\teta^{(1)}$ be any affine lift of $\eta^{(1)}$.
Let $\teta^{(\mu)}=\mu \star \teta^{(1)}$ where $\mu$ is a formal
parameter.
Let $X=\zeta z_P \in \C^g$
and $Y=XM$. Set $R_{\star}=\teta_{j_1}^{(\mu)} \in
\Alg(\mu)$, and let $R$
be the reduction of $R_{\star}$ modulo the equations in $\Alg$
coming from $\normalize(\teta^{(1)},\zeta\tildeP)$.
  \item
If 
$\ell_2 \neq 1$ is a sum of two squares, let $\eta^{(1)} = \sqrt{\ell_1} \ideta+\zeta P$ for $\zeta \in \Z$. Let 
$\teta^{(1)}$ be any affine lift of $\eta^{(1)}$.
Let $\teta^{(\mu)}=\mu \star \teta^{(1)}$ where $\mu$ is a formal
parameter, $\teta^{[1]} = \beta_0 \teta^{(\mu)}$.
Let $X=(\zeta z_P, \beta_0 \zeta z_P) \in (\C^g)^2$
and $Y=XM$. Set $R_{\star}=\teta_{j_1}^{(\mu)} \teta^{[1]}_{j_2} \in
\Alg(\mu)$, and let $R$
be the reduction of $R_{\star}$ modulo the equations in $\Alg$
coming from $\normalize(\teta^{(1)},\zeta\tildeP)$.
\item If 
$\ell_2$ is a sum of four squares, let $\eta^{(i)}=\sqrt{\ell_1} \ideta^{(i)}+\zeta_i P$ with
$\zeta_i \in \Z$ for
$i=1,2$ and fix corresponding affine lifts $\teta^{(i)}$. Fix also an
affine lift $\teta^{(12)}$ of $\eta^{(1)}+\eta^{(2)}$.
Let $\teta^{(\mu_1)}=\mu_1 \star \teta^{(1)}$,
$\teta^{(\mu_2)}=\mu_2 \star \teta^{(2)}$ and
$\teta^{(\mu_{12})}=\mu_{12} \star \teta^{(12)}$
where $\mu_1,\mu_2,\mu_{12}$ are formal parameters.
Let
$\teta^{[1]}=\alpha \teta^{(\mu_1)} + \beta \teta^{(\mu_2)}$ and $\teta^{[2]}=\gamma
\teta^{(\mu_1)} + \delta \teta^{(\mu_2)}$ in $\tildeA(\Kc(\mu_1,
\mu_2, \mu_{12}))$
(bootstrapping from $\teta^{(12)}$ using differential additions).
Let $X=(\zeta_1 z_P,
\zeta_2 z_P, (\alpha \zeta_1 + \beta \zeta_2) z_P, (\gamma \zeta_1 +
\delta \zeta_2) z_P)\in (\C^g)^4$ and $Y=XM$. Set $$R_{\star}= \teta^{(\mu_1)}_{j_1} \teta^{(\mu_2)}_{j_2}
\teta^{[1]}_{j_3} \teta^{[2]}_{j_4} \in \Kc(\mu_1, \mu_2, \mu_{12}).$$
Let $R$
be the reduction of $R_{\star}$ modulo the equations in $\Kc$
coming from 
$\normalize(\teta^{(1)},\zeta_1\tildeP)$,
$\normalize(\teta^{(2)},\zeta_2\tildeP)$ and
$\normalize(\teta^{(12)},(\zeta_1+\zeta_2)\tildeP)$.
\end{itemize}
We have that $R \in
\Kc$ and
\begin{equation}\label{eq:thetanull3}
  \theta_k^B(Y_1)\ldots \theta_0^B(Y_r)=\frac{1}{N}\sum_{T \in
  K^{t}(\overk)}
    R(T).
\end{equation}
\end{proposition}
\begin{proof}
  First, we prove the case
  $\ell_2 \neq 1$ is a sum of two squares.  In order to prove
  that $R$ is indeed in $\Alg$, because of Proposition
  \ref{prop:pseudogoodformal}, it suffices to show that the rational
  function $R_{\star}$ does not change when we act on $\teta^{(1)}$ by
  $\mu$ in an \'etale extension of $\Alg$ that satisfy the equations
  coming from $\normalize(\teta^{(1)},\tildeP)$. 
  
  For this let $\teta_{\theta}$ in $\tildeA(\Alg_{\C}(\mu))$ be such
  that $\teta^{(\mu)} = \mu' \star \teta_{\theta}$ where $\mu' \in
  \Alg_{\C}$ is a $\ell^{th}$-root of unity because of equation
  \eqref{eq:equationscompat2}.  Let $\teta^{[1]}_\theta
  =\scalarmult(\beta_0, \teta_\theta, \teta_\theta, \tildeO,\tildeO)$.
  By \cite[Rem. 3]{RobLub}, we have $\teta^{[1]} = {\mu'}^{\beta_0^2}
  \star \teta^{[1]}_\theta$. Thus, we have $R_*=\teta^{(\mu)}_{j_1}
  \teta^{[1]}_{j_2}=\mu' \teta_\theta(T) {\mu'}^{\beta_0^2}
  \teta^{[1]}_\theta$ with $\beta_0^2
  \equiv -1 \mod \ell$. We deduce that $R \in \Alg$ since it is left invariant by the
  action of $\mu'$.

  Now, let $\teta_{\theta}$ in $\tildeA(\Alg_{\C})$ be such that 
 $\{\teta_{\theta}(Q) \mid Q \in K(\overk)\}$ is a system of
  compatible lifts relative to $\tildeP$. 
  Using Theorem \ref{th:mainold}, we can suppose that $\teta^{(\mu)}=
  \teta_{\theta}$ modulo the equations coming from
  $\normalize(\teta^{(1)},\zeta\tildeP)$.  Then we have
  $R(T)=\teta_\theta(T) \teta_\theta(\beta_0 T)=
  \theta_{j_1}^A(X_1+z_{T}) \theta_{j_2}^A(X_2+\beta_0 z_{T})$ for
  $z_T \in \C^g$ such that $\rho_n(z_T)=T$.  Thus the relation
  (\ref{eq:thetanull3}) is just a consequence of Proposition
  \ref{prop:koigen}. The $\frac{1}{N}$ in front of the right hand side
  of \eqref{eq:thetanull3} comes from the fact that the
  parametrisation of $(\ker F)(\overk)$ by $K(\overk)$ we have used is
  an epimorphism the kernel of which has cardinality $N$.

  Now suppose that 
  $\ell_2$ is a sum of four squares.
  Fix a system of global compatible lifts $\{\alpha P+Q \mid \alpha \in
  \{0,\ldots,\ell-1\}, Q \in K(\overk) \}$ relative to $\tildeP$.
  For $i=1,2$, let
  $\teta^{(i)}_\theta$ be the formal point such that
  $\teta^{(i)}_{\theta}(Q_1,Q_2)$ is the corresponding compatible lift above
  $\zeta_i P+Q_i$, and define
  $\theta^{(12)}_{\theta}$ so that
  $\teta^{(12)}_{\theta}(Q_1,Q_2)$ is the corresponding compatible lift above
  $(\zeta_1+\zeta_2) P+Q_1+Q_2$.
 
  We use the same line as for the case above.
  Let $\mu'_1$, $\mu'_2$ and $\mu'_{12}$ be roots in $\Alg_{\C}$ of
  the equations $(E)$ given by $\normalize(\teta^{(1)},\zeta_1\tildeP)$,
  $\normalize(\teta^{(2)},\zeta_2\tildeP)$ and
  $\normalize(\teta^{(12)},(\zeta_1+\zeta_2)\tildeP)$.  By Proposition
  \ref{prop:pseudogoodformal}, $\mu'_1 \star \teta^{(1)}
  ,\mu'_2 \star \teta^{(2)}$ and $\mu'_{12} \star \teta^{(12)}$ differ
  from $\teta^{(1)}_{\theta}, \teta^{(2)}_{\theta}$ and
  $\teta^{(12)}_{\theta}$ by $\ell^{th}$-roots of unity $\mu_1$,
  $\mu_2$ and $\mu_{12}$ respectively.

  Let $\teta^{[1]}_{\theta}=\alpha \teta^{(1)}_{\theta} + \beta \teta^{(2)}_{\theta}$ and $\teta^{[2]}_{\theta}=\gamma
  \teta^{(1)}_{\theta} + \delta \teta^{(2)}_{\theta}$.
We can use repeatedly
\cite[Rem. 3]{RobLub} to 
obtain that: $\alpha
  \teta^{(\mu_1)}+\teta^{(\mu_2)}=\frac{\mu_{12}^\alpha \mu_1^{\alpha
(\alpha-1)}}{\mu_2^{\alpha-1}} (\alpha
\teta_\theta^{(1)}+\teta_\theta^{(2)})$, $\alpha
\teta^{(\mu_1)}=\mu_1^{\alpha^2} (\alpha \teta^{(1)}_\theta)$ and
finally $\teta^{[1]} = \mu_{12}^{\alpha\beta}
  \mu_1^{\alpha^2-\alpha\beta} \mu_2^{\beta^2-\alpha\beta}
  \,\teta^{[1]}_\theta$. We have in the same way
  $\teta^{[2]} = \mu_{12}^{\gamma\delta}
  \mu_1^{\gamma^2-\gamma\delta} \mu_2^{\delta^2-\gamma\delta}
  \,\teta^{[2]}_\theta$. Thus, for $T=(T_1,T_1) \in K^2(\overk)$, we have
  $R(T)=\teta^{(1)}_{j_1}(T) \teta^{(2)}_{j_2}(T)
  \teta^0_{j_3}(T)
  \teta^1_{j_4}(T)= \Delta \teta^{(1)}_{\theta,j_1}(T)
  \teta^{(2)}_{\theta,j_2}(T) \teta^0_{\theta,j_3}(T)
  \teta^1_{\theta,j_4}(T)$ where $$\Delta = \mu_{12}^{\alpha\beta+\gamma\delta}
  \mu_1^{1+\alpha^2+\gamma^2-\alpha\beta-\gamma\delta}
  \mu_2^{1+\beta^2+\delta^2-\alpha\beta-\gamma\delta}.$$ But an
  easy computation shows that $\alpha\beta+\gamma\delta=0$ and
  $\alpha^2+\gamma^2=\beta^2+\delta^2=\frac{b^2+c^2}{a^2+b^2}$ so
  that $\Delta=1$. Finally, we obtain that
  $R(T)=\theta_{j_1}(X_1+z_{T_1})\theta_{j_2}(X_2+z_{T_2})
  \theta_{j_3}(X_3+\alpha
  z_{T_1} + \beta z_{T_1}) \theta_{j_4}(X_4+\gamma z_{T_1} + \delta
  z_{T_2})$ for $z_{T_i} \in \C^g$ such that $\rho_n(z_{T_i})=T_i$ for
  $i=1,2$ and relation 
(\ref{eq:thetanull3}) is again a consequence of Proposition
\ref{prop:koigen}.

The simpler case $\ell_2=1$ can be treated in a similar manner as the
other cases. We leave it up to the reader.
\end{proof}

\onlongversion{%
\begin{remark}
  The fact that the $\ell^{th}$-roots of unity appearing in the
  evaluation of the right hand of the formula \eqref{eq:koizumi} of Proposition
  \ref{prop:koigen} cancel out is not a miracle. It can be explained with
  a more conceptual point of view: these $\ell^{th}$-roots of unity
  correspond to choices of a level $\ell n$-theta structure for $B$
  compatible with the level $n$-theta structure of $A$ via the
  contragredient isogeny of $f:A\rightarrow B$ and we can interpret
  the change of level formula (\ref{eq:thetanull1}) as "forgetting"
  the $\ell$-torsion part of these theta structures to recover a level
  $n$ theta structure for $B$ (see for instance \cite{MR2982438}).  
  
  The stronger fact that they also cancel out when only considering the
  terms in the sum comes from the fact that these terms already forget the part
  of the $\ell$-structure whose automorphisms acts by the $\star$
  operator
  \cite[Prop.~18]{MR2824556}.
\end{remark}}

  In order to turn Proposition \ref{prop:main1} into an algorithm, it
  remains to explain how to compute efficiently the right hand side of
  (\ref{eq:thetanull3}). This is done by the:
  
\begin{lemma}\label{lemma1}
  Let $W \in  \Alg=k[U]/(Q)$, let $(T,S) \in k[U]$ be respectively the
  quotient and remainder of the Euclidean division of $U W Q'$ by $Q$
  (where $Q'$ is the first derivative of $Q$).
We have:
\begin{equation}\label{eq:small}
  \sum_{P \in K(\overk)} W(P) =T(0).
\end{equation}
\end{lemma}
\begin{proof}
Let $\Rc$ be the set of roots of $Q$ in $\overk$, we want to prove
that $\sum_{a \in \Rc} W(a)=T(0)$. We have 
\begin{equation}\label{eq:small1}
  \frac{UW Q'}{Q}=\sum_{a \in \Rc} \frac{UW}{U-a}. 
\end{equation}
For $a \in \Rc$ let $T_a$ be the quotient of the
Euclidean division of $U W$ by $(U-a)$ so that we have $UW=T_a
(U-a)+a W(a)$. Putting this in (\ref{eq:small1}), we obtain that
$\frac{UWQ'}{Q}=\sum_{a \in \Rc} (T_a + \frac{aW(a)}{U-a})$. As $Q
\sum_{a \in \Rc} \frac{a W(a)}{U-a}$ is an element of $k[U]$ of degree
$< \deg(Q)$, we deduce that $S=Q
\sum_{a \in \Rc} \frac{a W(a)}{U-a}$ and $T=\sum_{a \in \Rc} T_a$.
Moreover $T(0)= \sum_{a\in \Rc} T_a(0)$ but $-a
T_a(0)+aW(a)=0$ so that $T_a(0)=W(a)$ (it is easy to check that this also
holds if $a=0$) and we are done.
(Over $\C$, this lemma is just an easy application of the residue theorem.)
\end{proof}

As $R$ defined in Proposition \ref{prop:main1} is an element of
$\Alg^{\otimes r/2}$ if $\ell_2\neq 1$ and an element of $\Alg$ else, Lemma
\ref{lemma1} shows that, in the case 
$\ell$ is a sum of two squares
the
evaluation of the right hand of (\ref{eq:thetanull3}) can be done with
an Euclidean division of an element of $k[U]$ of degree bounded by
$\ell^g$ at the expense $\sO(\ell^g)$ operations in $k$. If 
$\ell$ is a sum of four squares,
$R$ is an element of $\Alg^{\otimes 2}$ and we can
resort two times to Lemma \ref{lemma1}, to carry out the evaluation of
(\ref{eq:thetanull3}). The dominant step is an Euclidean division of
an element of $\Alg[U]$ of degree bounded by $\ell^g$ which can be
done in $\sO(\ell^{2g})$ operations in $k$.  We call $\evaluate$ an
algorithm which takes as input $R \in \Alg^{\otimes r/2}$ and returns
$\sum_{T \in K^{r/2}(\overk)} R(T)$.

If we apply Proposition \ref{prop:main1}, with $P=0$, we obtain an
algorithm to compute $\theta^B_k(0) \ldots \theta^B_0(0)$ for any $k
\in \Zn$ that gives the projective theta null point of $B$ associated
to $\ell \Omega$. Let $P \in A(k)$ and let $z_P \in \C^g$ be such that
$\rho_n(z_P)=P$:
\begin{itemize}
  \item If $\ell$ is a square, Proposition 
\ref{prop:main1} with $\zeta=\sqrt{\ell}$ gives an expression for
    $\theta_k^B(\ell z_P) \theta_0^B(0)$ for $k \in \Zn$.
  \item  
    If $\ell_2=a^2+b^2$, Proposition
    \ref{prop:main1} with $\zeta=\sqrt{\ell_1} a$ gives an expression for
    $\theta_k^B(\ell z_P) \theta_0^B(0)$ for $k \in \Zn$.  
  \item If 
    $\ell_2=a^2+b^2+c^2+d^2$, Proposition
    \ref{prop:main1} with $\zeta_1=\sqrt{\ell_1} a$ and
    $\zeta_2=\sqrt{\ell_1} b$ gives an
    expression for $\theta_k^B(\ell z_P) \theta_0^B(0)^3$ for $k \in
    \Zn$.
\end{itemize} 
  
  In any cases, we obtain the projective coordinates for
  $f(P)$ and we have proved Theorem \ref{main}.
  Algorithms~\ref{algo:iso1} and~\ref{algo:iso2} recapitulate the isogeny computation algorithms we
  have described (we leave up to the reader the easy adaption in the case that $\ell$ is a
  square).  To give more details, in Algorithm~\ref{algo:iso1}
  the line $\teta^{[1]} \la \beta_0\teta^{(\mu)}$ is simply computed
  as $\scalarmult(\beta_0,\teta^{(\mu)}, \teta^{(\mu)},
  \tildeO,\tildeO)$.  Likewise, in Algorithm~\ref{algo:iso2}, the line
  $\teta^{[1]} \la \alpha \teta^{(\mu_1)}+\beta \teta^{(\mu_2)}$ is
  computed as $\alpha \teta^{(\mu_1)}+ \teta^{(\mu_2)} \la
  \scalarmult(\alpha, \teta^{(12)}, \teta^{(\mu_1)},\teta^{(\mu_2)},
  \tildeO)$, $\teta^{[1]} \la \scalarmult(\beta, \alpha
  \teta^{(\mu_1)}+ \teta^{(\mu_2)} , \teta^{(\mu_2)},\alpha
  \teta^{(\mu_1)} , \tildeO)$ and similarly for $\teta^{[2]}$.

\begin{algorithm}
\SetKwInOut{Input}{input}\SetKwInOut{Output}{output}
\SetKwComment{Comment}{/* }{ */}
\Input{
  \begin{itemize}
    \item $\tildeO^A$ the level $n$ theta null point of $(A,\pol_0)$ associated
      to $\Omega \in \Sie$;
     \item $\ell \in \N$ such that 
       $\ell=\ell_1
       \ell_2$ where $\ell_1$ is the biggest square factor of $\ell$ and a decomposition
       $\ell_2=a^2+b^2 \neq 1$ ;
     \item $Q \in k[U]$ such that $\deg(Q)=\ell^g$ describing $K$ (see
       Section \ref{sec:gener});
     \item $P \in A(\overk)$ given its projective coordinates;
      \item $k \in \Zn$.
  \end{itemize}
}
\Output{
  $(f(P))_k$ the $k^{th}$ projective coordinates associated to the level $n$
  projective embedding of $B$
provided by $\ell \Omega$.}
\BlankLine
$\eta \la \normaladd(\sqrt{\ell_1} \ideta ,\sqrt{\ell_1} aP)$ where 
$\sqrt{\ell_1} \ideta$ and
$\sqrt{\ell_1}a P$ are
computed with $\scalarmult$\;
$\teta^{(\mu)} \la \mu \star \teta$ where $\mu$ is a formal parameter\;
$\teta^{[1]} \la \beta_0\teta^{(\mu)}$
where
$\beta_0 = - b/a \mod \ell$\;
$R \la \teta^{(\mu)}_{j_1} \teta^{[1]}_{j_2} \mod \normalize(\teta,\sqrt{\ell_1}a \tildeP)$ where $j=(k,\ldots, 0)M^{-1}
\in \Zn^r$\;
\Return $\evaluate(R)$\;
\caption{Algorithm $\genericisogeny$
for $\ell$ a sum of two squares
}
\label{algo:iso1}
\end{algorithm}

\begin{algorithm}
\SetKwInOut{Input}{input}\SetKwInOut{Output}{output}
\SetKwComment{Comment}{/* }{ */}
\Input{
  \begin{itemize}
    \item Same as in Algorithm~\ref{algo:iso2} except that
     we have a decomposition
       $\ell_2=a^2+b^2+c^2+d^2$ ;
  \end{itemize}
}
\Output{
  $(f(P))_k$ the $k^{th}$ projective coordinates associated to the level $n$
  projective embedding of $B$
provided by $\ell \Omega$.}
\BlankLine
$\eta^{(i)} \la \normaladd(\sqrt{\ell_1} \ideta^{(i)} ,\zeta_i P))$ where 
$\zeta_1=\sqrt{\ell_1} a$, $\zeta_2=\sqrt{\ell_1} b$ for $i=1,2$ and
$\sqrt{\ell_1} \ideta^{(i)}$ and $\zeta_i P$ are computed with
$\scalarmult$\;
$\eta^{(12)} \la \normaladd(\eta^{(1)}+\eta^{(2)})$\;
$\teta^{(\mu_1)} \la \mu_1 \star \teta^{(1)}$, 
$\teta^{(\mu_2)} \la \mu_2 \star \teta^{(2)}$, 
$\teta^{(\mu_{12})} \la \mu_{12} \star \teta^{(12)}$
 where $\mu_1, \mu_2, \mu_{12}$ are formal parameters \;
$(E) \la \normalize(\teta_1,\zeta_1 \tildeP) \cup
\normalize(\teta_2,\zeta_2 \tildeP) \cup 
\normalize(\teta_{12},(\zeta_1+\zeta_2) \tildeP)$
\;

$\teta^{[1]} \la \alpha \teta^{(\mu_1)}+\beta \teta^{(\mu_2)}$,
$\teta^{[2]} \la \gamma \teta^{(\mu_1)}+\delta \teta^{(\mu_2)}$ where
$\left( \begin{smallmatrix}
\alpha & \beta \\ \gamma & \delta \end{smallmatrix} \right) = 
\left( \begin{smallmatrix} a & b \\ -b & a \end{smallmatrix}
\right)^{-1}
\left( \begin{smallmatrix} c & d \\ d & -c \end{smallmatrix} \right)$\;

$R\la \teta^{(\mu_1)}_{j_1} \teta^{(\mu_2)}_{j_2}
\teta^{[1]}_{j_3} \teta^{[2]}_{j_4}$  modulo $(E)$ where $j=(k,\ldots, 0)M^{-1}
\in \Zn^r$\;
\Return $\evaluate(R)$\;
\caption{Algorithm $\genericisogeny$
for $\ell$ a sum of four squares}
\label{algo:iso2}
\end{algorithm}

 We note that in Algorithm~\ref{algo:iso2} we need three calls to
 \normalize, which costs each two scalar multiplications by $\ell$. We
 can improve this as follow: first compute $\eta^{(1)}_0 =
 \ideta^{(1)}+P$ and use \normalize{} to normalize an affine lift
 $\teta^{(1)}_0$ up to a factor $\mu_1$. From this data it is easy to
 recover a compatible affine lift $\teta^{(2)}_0$ of $\ideta^{(2)}+
 P$ up to a factor $\mu_2$.  Using differential additions, one can
 then recover the compatible lifts $\teta^{(1)}$, $\teta^{(2)}$ of
 Proposition~\ref{prop:main1}.  Likewise one can normalize
 $\ideta^{(1)}+\ideta^{(2)}$ up to a factor $\lambda_{12}$ by using
 only equations~\eqref{eq:equationscompat1}.  One can compute
 $\teta^{(12)}$ using differential additions from a compatible affine
 lift of $P+\ideta^{(1)}+\ideta^{(2)}$. But the latter point can be
 computed as a three way addition \cite[Section~3.6]{optimal} between
 the corresponding lifts above of $P+\ideta^{(1)}, P+\ideta^{(2)},
 \ideta^{(1)}+\ideta^{(2)}$. This method requires only three scalar
 multiplications by $\ell$ rather than six to normalize the points.

Throughout the article we have supposed that $4 \mid n$ so that we can
compute normal additions. But actually one can work with level $n=2$
exactly as in \cite[Section~4.4]{robertcosset}.
The only difficulty is the case
when $\ell$ is a sum of four squares
where we need to
compute a point of the form $\eta^{(1)}+\eta^{(2)}$ in
Proposition~\ref{prop:main1}.
Here, we replace the normal addition of
the two formal points $\eta^{(1)}$ and $\eta^{(2)}$ by any formal point $\eta$ such
that $\eta(Q)\in \{\eta^{(1)}(Q)+\eta^{(2)}(Q),
\eta^{(1)}(Q)-\eta^{(2)}(Q)\}$.
Computing such an $\eta$ requires taking a certain square root in $\Alg$ as in
\cite[Section~3.3]{optimal} 
(this square root
may have more than 2 solutions since $\Alg$ may not be a field).
Also $\Alg$ is not an \'etale algebra anymore because we identify $P$ with
  $-P$ so the points in $K \setminus \{0_A\}$ have multiplicity~$2$, but we
can instead work directly on $(K \setminus \{0_A\})/\pm 1)$ and gain a
factor~$2$. See Section~\ref{sec:examples} for an example.

\onlongversion{%
We end up this section by comparing the algorithm presented in this
paper, with the algorithm of \cite{robertcosset}. In the later
algorithm, all the affine lifts of points of $K(\overk)$ are globally
compatible in the sense that they are deduced by the way of differential
additions from the knowledge of a minimal set of compatible good lifts
as in Theorem~\ref{th:mainold}.
This
property is not true if we specialize the lifts of the formal points
of Proposition \ref{prop:main1} to geometric points of $K$. 
For instance in the case $\ell \equiv 1\mod 4$, $\ell$ prime 
then for a fixed $Q \in K(\overk)$
the lifts $\teta^{(\mu)}(Q), \teta^{[1]}(Q) \in
\tildeK(\overk)$ (modulo the equations coming from $\normalize$) of $\zeta P+Q$ and $\beta_0 (\zeta P+Q)$ are 
locally compatible with $\tildeP$, but they may not be globally compatible
over all $Q \in K(\overk)$.
But as seen in the proof of Proposition~\ref{prop:main1},
in the evaluation of the right hand
of (\ref{eq:thetanull1}) this local compatibility between the elements
appearing in the sum is enough.
The authors of \cite{robertcosset} did not use the local compatibility,
because it is actually faster to compute potential compatible lifts for a basis
of $K(\overk)$ and use differential addition to get the other compatible points
than it is to normalize locally for each term in the sum.
This show the usefulness of working with formal points since we can
normalize everything once and for all.}

\section{Example}\label{sec:examples}

We give a simple example to illustrate the algorithm in the case that
the dimension $g=1$, the base field $k=\F_{1009}$ and the level $n=2$.
Let $\tilde{0}_A \in \Aff^{\Z(2)}$ be the level~$2$ (affine) theta null point
with coordinates $(971,94)$. This theta null point corresponds to the
elliptic curve $A$ with Weierstrass equation $y^2 = x^3 + 762x^2 +
246x$. This elliptic curve has a unique subgroup $K$ defined over $k$ in its
$5$-torsion. If $(U_0,U_1)$ are the theta coordinates of level $2$, a
system of equations for this kernel is given by: $U_0=1$ and $R(U_1)=0$ where
$R(U_1)=U_1^5 + 751U_1^4 + 546U_1^3 + 447U_1^2 + 660U_1 + 339$. We
explain how to compute the level $2$ theta null point of $B=A/K$.

The polynomial $R$ factorizes as $R(U_1)=(U_1 + 268)Q(U_1)^2$ where
$Q(U_1)=U_1^2 + 746U_1 + 353$. The linear term corresponds to the $U_1$
coordinate of the theta null point of $A$, while the fact that $Q$ has
multiplicity
two comes from the fact that we work on the Kummer
variety here associated to $A$ (see Section \ref{sec:basicfacts}). 

We consider the algebra $\Alg=k[U]/(Q)$, and we
look at the formal point $\eta=(1:U)$. 
Let $\teta=(\lambda,\lambda U)$ be a potential compatible lift.
An easy computation shows that
$2 \teta=\lambda^4(980U+906,103U+7)$ and
$3\teta=\lambda^9(861U+437,572U+129)$. We thus find that $\lambda^5=126U+129=
(980U+906)/(861U + 437)=(103U + 7)/(572U + 129)$.

Now from Equation~\eqref{eq:thetanull1}, we have (up to a common factor) that
\begin{multline*}
\theta_{i_1}^B(0)\theta_{i_2}^B(0)=
\sum_{\substack{t_1,t_2 \in K\\ t_1+2t_2=0\\ -2t_1+t_2=0}} \theta_{i_1}^A(t_1)\theta_{i_2}^A(t_2)
=\sum_{t \in K} \theta_{i_1}^A(t)\theta_{i_2}^A(2t)\\
=\theta_{i_1}^A(0)\theta_{i_2}^A(0)+
\sum_{t \in K \setminus \{0\}} \theta_{i_1}^A(t)\theta_{i_2}^A(2t).
\end{multline*}
Now if we let $W=\theta_{i_1}(\teta)\theta_{i_2}(2\teta)$, we have
that 
$\theta_{i_1}^B(0)\theta_{i_2}^B(0)=
\theta_{i_1}^A(0)\theta_{i_2}^A(0)+2T(0)$ where $T$ is the polynomial
defined in Lemma~\ref{lemma1}.

If $i_1=i_2=0$ then $W=\lambda^5(980u+906)$, $T=380$ and 
$\theta_{i_1}^B(0)\theta_{i_2}^B(0)=186$.
If $i_1=0,i_2=1$ then $W=\lambda^5(103u+7)$, $T=629U+529$ and 
$\theta_{i_1}^B(0)\theta_{i_2}^B(0)=513$.

The level $2$ theta null point $(186:513)$ corresponds to the elliptic
$B$ given by the Weierstrass equation
$y^2 = x^3 + 133x^2 + 875x$. In this case we could have computed the
isogeny by an application of V\'elu's formulas, this indeed yields a
curve isomorphic to $B$. (The conversion between theta and Weierstrass
coordinates was done using \cite{avisogenies}).

\section{Conclusion}

In this paper, we have presented an algorithm to compute isogenies
between abelian varieties in the arguably most general setting which
takes advantage of the field of definition of the kernel in order to
improve the complexity. We note that the quasi-optimality announced in
the title is only for the case that $\ell$ is a sum of two squares. It would be very interesting
to extend it to encompass all cases.  Another question is  to
handle the case $\ell$ non prime to $2n$.  The problem here is
that there may be several ways to descend a symmetric theta structure
of $\pol^{\ell}$ along the isogeny defined by $K$, and it is not clear
how to specify the choice of a symmetric theta structure of level~$n$
on $B=A/K$ via a set of equations.

A related question is how to find equations for a rational kernel $K$
given the lack of modular polynomials for higher dimension.  In the
case that $A$ is a Jacobian of a curve over a finite field the zeta
function of which is known, it is possible to work with the geometric
points of $\ell$-torsion by taking random points in an appropriate
extension (for more details see \cite{avisogenies}). One can
then try to find directly a basis of a rational kernel $K$.
Generating the equations of $K$ from such a basis takes $\sO(\ell^g)$
operations in the field $k'$ the compositum of the field of
definitions of the geometric points of $K$.  If we already have the
geometric points of the kernel, it might seem faster to directly use
the algorithm from \cite{robertcosset}, but actually when $\ell$ is a
sum of four squares this algorithm takes $\sO(\ell^{2g})$
operations in $k'$, so it is slower than the algorithm presented in
this paper ($\sO(\ell^g)$ operations in $k'$ to find the equation of
the kernel and $\sO(\ell^{2g})$ operations in $k$ to compute the
isogeny).

\bibliographystyle{abbrv}
\bibliography{rational}
\end{document}